\documentclass [10pt, leqno]{amsart}
\usepackage {graphicx}
\usepackage{amsmath,amssymb, amsthm}

\numberwithin{equation}{section}

\newtheorem{theorem}{Theorem}[section]
\newtheorem*{theorem*}{Theorem}
\newtheorem{proposition}[theorem]{Proposition}
\newtheorem{lemma}[theorem]{Lemma}
\newtheorem{corollary}[theorem]{Corollary}

\newtheorem{remark}[theorem]{Remark}

\def\D{\mathbb{D}}
\def\C{\mathbb{C}}
\def\N{\mathbb{N}}
\def\Z{\mathcal{Z}}
\def\F{\mathcal{F}}
\def\E{\mathbf{E}}
\def\P{\mathbf{P}}

\def\S{\overset{\mbox{}\,\,\,\scriptscriptstyle \prime\prime}{\mathcal{S}}}

\begin{document}

\title[Zeros of sections of power series]{Zeros of sections of power series:\\ deterministic and random}
 \author{Jos\'e L. Fern\'andez}
 \thanks{Research partially supported by Fundaci\'on Akusmatika.}
 \address{Departamento de Matem\'aticas, Universidad Aut\'onoma de Madrid, 28770 Madrid, Spain.}
 \email{joseluis.fernandez\@@uam.es}

 \keywords{Power series, random power series, zeros, sections, Jentzsch--Szeg\H{o} theorem, equidistribution.}
 \subjclass[2010]{30B20,  30B10}
 \date{\today}

 \maketitle

 \begin{center}
 \textsc{\footnotesize To Juha Heinonen, in memoriam}
 \end{center}

 \begin{abstract}
{We present a streamlined proof (and some refinements) of a
characterization
 (due to F. Carlson and G. Bourion, and also to P. Erd\H{o}s and H. Fried) of  the so called Szeg\H{o} power series. This characterization is then applied to readily obtain some (more) recent
known results and some new results on the asymptotic distribution of
zeros of sections of  random power series, extricating quite
naturally the deterministic ingredients. Finally, we study the
possible limits of the zero counting probabilities of a power
series.
 }
 \end{abstract}

\

\section{Introduction}

The first aim of this paper is to present a streamlined proof and a
refined version of a characterization
 (due to F. Carlson and G. Bourion, and also to P. Erd\H{o}s and H. Fried) of
 the so called Szeg\H{o} power series: theorems
\ref{theor:CarlsonBourion} and \ref{theor:gauge and convergence of
probabilities}.

 That characterization is then applied in Section 5 to readily obtain some (more) recent
known results and some new results on the
asymptotic distribution of zeros of sections of  random power
series, extricating quite naturally the deterministic ingredients. Finally, in Section 6 we study the
possible limits of the zero counting probabilities associated to a power
series.

\smallskip

We shall denote by $\F$ the class of power series  whose radius of convergence is
$1$. The results which we are about to discuss concerning such $f$
can be translated, with  obvious scaling, to power series of positive and finite radius of convergence.

\smallskip

For a given power series $f \in \F$ and for each $n \ge 0$, we
denote by $s_n=s_n(f)$ the $n$-th section of the power series:
$s_n(z)=\sum_{k=0}^n a_k z^k\, ,$ and by $\Z_n$ the (multi-)set of
the  zeros of~$s_n$. To each non constant $s_n$ we associate two measures: we denote  by
$\mu_n=\mu_n(f)$ the zero counting measure
$$
\mu_n=\frac{1}{n} \sum_{w\in \Z_n} \delta_w\, ,
$$
a weighted sum of Dirac deltas placed at the zeros of $s_n$ repeated
according to their multiplicity, and we denote by $\rho_n=\rho_n(f)$ the
circular projection of $\mu_n$:
$$
\rho_n=\frac{1}{n} \sum_{w\in \Z_n} \delta_{|w|}\, .
$$If $a_n\neq 0$, then $\mu_n$ and $\rho_n$ are probability measures. If $a_n=0$ (and $s_n$ is non constant), we append the definition above by adding a Dirac delta at $\infty_{\scriptscriptstyle \C}$ with  mass $n-\textrm{deg}(s_n)$ so that $\mu_n$ and $\rho_n$ become  probability measures on the Riemann sphere $\widehat{\C}$.  By $F_n$ we denote the distribution function of $\rho_n$, given by
$$
F_n(t)=\mu_n\big(|z|\le t\big)\,, \quad \mbox{for $t\ge 0$}\, ,
$$
thus $F_n(t)$ is the average number of zeros of $s_n$ within the disk $\{z\in \C: |z|\le t\}$.

By Hurwitz's theorem, $\lim_{n \to \infty} F_n(t)=0$ for any $t<1$;
actually, $F_n(t)=O(1/n)$,  for any fixed $t<1$.

\smallskip

We are concerned in this paper with the convergence as $n \to
\infty$ of the probabilities $\mu_n$ and $\rho_n$ associated to a
given $f\in \F$ and with the potential limits which these
probability measures may have. By convergence we mean  weak
convergence, so that a sequence $(\lambda_n)_{n \ge 0}$ of measures
on $\widehat{\C}$ converges to a measure $\lambda_\infty$ if
$$
\lim_{n\to \infty} \int h \, d\lambda_n=\int h \, d\lambda_\infty\, ,
\quad \mbox{for any function $h$ bounded and continuous on $\widehat{\C}$}\,
.
$$

We shall denote by $\Lambda$ the uniform probability (normalized
Lebesgue measure) on $\partial \D$. Convergence of the zero counting
probabilities $\mu_n$ to $\Lambda$ means that for each $h$ as above,
$$
\lim_{n \to \infty} \frac{1}{n} \sum_{w \in \mathcal{Z}_n}
h(w)=\frac{1}{2\pi}\int_{0}^{2 \pi} h(e^{\imath \vartheta})\, d
\vartheta\, .$$

\smallskip
In Section 2, we shall discuss the main results about the asymptotic
behavior of $\mu_n$ and~$\rho_n$ for a given $f\in \F$. Section 3 is
devoted to results connecting coefficients and zeros of
polynomials, and some proofs. Section 4 discusses the aforementioned
streamlined proof of the characterization of the Szeg\H{o} class.
This characterization is then applied in Section~5 to the study of
the sequences of zero counting measures of random power series.
Finally, returning to the deterministic context, Section 6 discusses
the possible limits of the zero counting measures of a given $f$ and
exhibits an example of a power series whose sequence of $\rho_n$ is
dense.

\section{Asymptotics of zero counting measures}\label{section:general asymptotic results}

Here we describe the main results about asymptotics of zero counting
measures, from the seminal work of Jentzsch and Szeg\H{o}
up to the characterization of those holomorphic functions $f$ whose
zero counting measures converge to the uniform probability
$\Lambda$.

\begin{theorem}\label{theor:JentzschSzego}

\mbox{}\smallskip

\indent {\rm i)} For any $f\in \F$,
there is a subsequence $(n_k)_{k \ge 1}$ such that $\rho_{n_k} \
\mbox{converges to} \ \delta_1$.

\smallskip

\indent {\rm ii)}  Given $f\in \F$, if for
a subsequence $(n_k)_{k \ge 1}$ the probability measures
$\rho_{n_k}$ converges to~$\delta_1$, then $\mu_{n_k}$ converges to
$\Lambda$, and conversely.
\end{theorem}

Part i) of Theorem  \ref{theor:JentzschSzego} is due to
Jentzsch, \cite{Jentzsch};
it claims that there is a subsequence $\mu_{n_k}$ of  the $\mu_n$
asymptotically concentrated on $\partial \D$. Since $\lim_{k \to
\infty} F_{n_k}(t)=0$ for each $t <1$, the conclusion of part $i)$
is equivalent to the statement that $\lim_{k \to \infty}
F_{n_k}(T)=1$, for each $T>1$.


\smallskip

The second part,  ii), of Theorem
\ref{theor:JentzschSzego}, which is due to Szeg\H{o},
\cite{Szego},  says that just simple radial concentration of the
mass of $\mu_n$ towards  the unit circle $\partial \D$ is equivalent
to the (much more precise) statement that the $\mu_n$ converge to the
uniform probability $\Lambda$ on $\partial \D$.

\smallskip

The paradigmatic example of Theorem \ref{theor:JentzschSzego}
is the power series ${1}/(1-z)=\sum_{k=0}^\infty z^k$. In this case
$\Z_n$ consists of  the $(n{+}1)$-th roots of unity except $z=1$, and
the whole sequence~$\mu_n$ converges to $\Lambda$.
By contrast, for the lacunary power series $f(z)=\sum_{k=0}^\infty
z^{2^k}$ a simple application of Rouch\'e's theorem gives that
$\mu_{2^k}$ converges to $\Lambda$, while,  since $s_{2^k-1}\equiv
s_{2^{k-1}}$, the probabilities $\rho_{2^{k-1}}$ converge to
$(\delta_1+\delta_{\infty})/2$;  the whole sequence of zero
counting measures of $f$ does not converge.

\smallskip

For a general treatment and a modern account of the theory of asymptotic distribution of zeros of polynomials we would like to refer to \cite{AndrievskiiBlatt}; the reader will find there complete references to the many authors who have contributed to the subject.

\smallskip

We say that a power series $f\in \F$ is a \textbf{Szeg\H{o} power
series} (and belongs to the \textbf{Szeg\H{o} class} $\S$) if the corresponding \textit{complete} sequence of zero counting
measures $\mu_n$ converges to the measure $\Lambda$.

Naturally, we would like to have conditions on the coefficients
$a_n$ of the power series $f$ which would imply that $f$ is a
Szeg\H{o} power series. Szeg\H{o} gave in  \cite{Szego} one first
such condition:

\begin{theorem}[Szeg\H{o}]\label{theor:Szego} \
If
\begin{equation}
\label{eq:szegos condition}
\lim_{n \to \infty} \sqrt[n]{|a_n|}=1\, ,
\end{equation}
then $f \in \S$.
\end{theorem}

Condition \eqref{eq:szegos condition} is quite restrictive: for each integer $N \ge 2$, the  power series $1/(1-z^N)=\sum_{k=0}^\infty z^{kN}$ belongs to $\S$, but $\liminf_{n \to \infty}\sqrt[n]{|a_n|}=0$.

\smallskip

Theorem \ref{theor:CarlsonBourion} below is Carlson's
\textit{characterization} (in terms of the coefficients $a_n$) of the Szeg\H{o}
class $\S$; to state it we need to introduce a few concepts and some
further notation.

\medskip

\subsection{Gauge and index of power series}\label{sect:gauge and index} Consider a power series
$f(z)=\sum_{n=0}^\infty a_n z^n \in \F$. For each $\gamma \in [0,1)$, define $$
A_n(\gamma)=\max_{(1-\gamma) n \le k\le n}|a_k|\, , \quad \mbox{for each $n \ge 0$}\, ,$$
and
$$
L(\gamma)=\liminf_{n \to \infty} \sqrt[n]{A_n(\gamma)}\, .
$$
Observe that
$L(\gamma)$ increases with $\gamma$
and that $0 \le L(\gamma)\le 1$, for $\gamma \in[0,1)$.

\smallskip

We define the \textbf{index} $\Gamma$ of a power series $f\in \F$ as
$$
\Gamma:=\inf\{\gamma \in (0,1):  L(\gamma)=1\}\, .
$$
We set $\Gamma=1$ if $L(\gamma)<1$ for each
$\gamma\in (0,1)$; this occurs, for instance, for $\sum_{k=0}^\infty
z^{k!}$.

\smallskip

The \textbf{gauge} $G$ of a power series $f \in \F$ is defined by
$$
G:=\lim_{\gamma \downarrow  0} L(\gamma)=\inf_{\gamma \in (0,1)} L(\gamma)\, .
$$
Observe that $0 \le G \le 1$ and that gauge $G=1$ is equivalent to index $\Gamma=0$.

\smallskip

A related but different notion of \lq\lq gap index\rq\rq appears in \cite{Rosembloom}, page 277.

\smallskip

Lacunary series like $\sum_{k=0}^\infty z^{q^k}$, where
$q$ is an integer, $q \ge 2$, have  index $\Gamma=1-1/q$ and gauge
$G=0$. More generally,

\begin{lemma}
For any $t \in (0,1]$ and any $g \in [0,1)$,  there exists a power series $\sum_{n=0}^\infty a_n z^n \in \F$ with index $\Gamma=t$ and gauge $G=g$.
\end{lemma}

This lemma, combined with the observation that $G=1$ is equivalente to $\Gamma=0$, means that $\{(\Gamma,G): \, f\in \F\}=(0,1]\times [0,1) \cup \{(0,1)\}$.

\begin{proof}
Let $(m_k)_{k \ge 1}$ be an increasing sequence of positive integers such that $m_{k}/m_{k+1}\to 1-t$, as $k \to \infty$. Denote by $\mathcal{M}$ the set $\mathcal{M}=\{m_k: k \ge 1\}$.

Define the coefficient sequence $(a_n)_{n \ge 0}$ by $a_n=1$ if $n\in \mathcal{M}$ and $a_n=g^n$ otherwise. Observe that $\sum_{n=0}^\infty a_n z^n \in \F$.

For each $n \ge 0$ and every $\gamma \in (0,1)$ we have that
$$\sqrt[n]{A_n(\gamma)}=\begin{cases}
1\, , &\mbox{if  $[n(1-\gamma),n]\cap \mathcal{M}\neq\emptyset$}\,,\\[3pt]
g^{\lceil n(1-\gamma)\rceil/n}\, , &\mbox{if  $[n(1-\gamma),n]\cap\mathcal{M}= \emptyset$}\, .\\
\end{cases}$$
Observe that $L(\gamma)\ge g^{1-\gamma}$, for any $\gamma \in(0,1)$.

Next we show that
$$L(\gamma)=\begin{cases}
1\, , &\mbox{if $1>\gamma>t$}\, ,\\[3pt]
g^{1-\gamma}\, ,&\mbox{ if $0<\gamma<t$}\, .
\end{cases}
$$

Let $1>\gamma>t$. If $m_k <n \le m_{k+1}$ then $n(1-\gamma)\le m_{k+1}(1-\gamma)<m_k$, for $k$ is large enough. Thus $\sqrt[n]{A_n(\gamma)}=1$, for $n$ large enough, and so $L(\gamma)=1$.

Let $0<\gamma<t$. Since $(1-\gamma)(m_{k+1}-1)>m_k$, for $k$ large enough, we have that  $\sqrt[n]{A_n(\gamma)}=g^{\lceil n(1-\gamma)\rceil/n}$, for $n=m_{k+1}-1$ and $k$ large enough. Therefore, $L(\gamma)\le g^{1-\gamma}$ and $L(\gamma)=g^{1-\gamma}$.

We conclude that  $\Gamma=t$, since $g<1$, and that $G=\lim_{\gamma\downarrow 0}L(\gamma)= g$, since  $t>0$.
\end{proof}

\begin{remark}
[Index and Ostrowsky gaps] \rm Power series $f$ with positive index $\Gamma >0$ are said to have
Ostrowsky (or Hadamard--Ostrowsky) gaps. This notion appeared first
in Ostrowsky characterization of overconvergent power series: a
power series $f\in \F$, analytically continuable beyond the unit
circle $\partial \D$, is overconvergent if and only if it has index
$\Gamma
>0$.\end{remark}

\begin{remark}[Gauge and index of rational functions]\rm
For $f(z)=1/(1-z^N)$, with $N\ge 2$, one has $L(0)=0$, but $L(\gamma)=1$ for each $\gamma >0$, and so $f$ has index $\Gamma=0$ and gauge $G=1$.

In general, any power series $f \in \F$ which defines a rational function
has index $\Gamma=0$. For, let $f(z)=\sum_{k=0}^\infty a_k
z^k=P(z)/Q(z)$, for each $z \in \D$, where $P,Q$ are relatively
prime polynomials and let $m$ be the degree of the denominator.
Consider $$ \alpha_n=\max\{|a_n|, |a_{n-1}|, \ldots, |a_{n-m+1}|\}\,
, \quad \mbox{for $n \ge m-1$}\, .
$$
A result of P\'olya (\cite{Polya}, Hilfssatz III, and also, \cite{PolyaSzego}, problem 243) gives  that
$$
\lim_{n \to \infty} \sqrt[n]{\alpha_n}=1\, .
$$
Note that we have \lq$\lim$\rq \ above,  not just \lq$\limsup$\rq. \ For any $\gamma \in (0,1)$ we have that
$$
A_n(\gamma) \ge \alpha_n\, , \quad \mbox{for $n \ge m/\gamma$}\,;
$$
consequently,
$$
L(\gamma)=\liminf_{n \to \infty}\sqrt[n]{A_n(\gamma)} \ge \liminf_{n \to \infty}\sqrt[n]{\alpha_n}=1\, ,
$$
so that $\Gamma=0$.
\end{remark}

\subsection{Characterization of the Szeg\H{o} class}\label{sect:characterization of SZ}

The following  characterization of Szeg\H{o} power series was
announced
by F. Carlson, \cite{Carlson}. A complete proof appeared in the
monograph \cite{Bourion}, page 19, of G. Bourion; later on, P.
Erd\H{o}s and H. Fried, \cite{ErdosFried}, gave an alternative proof
of the characterization (and credit Theorem
\ref{theor:CarlsonBourion} to Bourion).

\begin{theorem}[Carlson-Bourion]\label{theor:CarlsonBourion} Let
$f$ be a power series in $\F$. Then $$f \in \S \ \mbox{if and only
if \, its gauge $G$ is  $1$}\,.$$
\end{theorem}

We shall derive this theorem from the following:

\begin{theorem}\label{theor:gauge and convergence of probabilities}
Let $f$ be a power series in $\F$ with gauge $G$.\newline Then
\begin{equation}\label{eq:no mass beyond 1/G}
\liminf\limits_{n \to \infty} F_n(T)\ge 1-\frac{\ln(1/G)}{\ln(T)}>0\,, \quad \mbox{for each $T>1/G$}\, .
\end{equation}
And also,
\begin{equation}\label{eq:some mass before 1/G}
\liminf_{n \to \infty} F_n(T)<1\,, \quad \mbox{for each $T<1/G$}\, .
\end{equation}
\end{theorem}

Recall, in addition,  that for any $t<1$ we always have that $\lim_{n \to \infty} F_n(t)=0$.

\smallskip

Theorem \ref{theor:CarlsonBourion} may be derived from Theorem
\ref{theor:gauge and convergence of probabilities} as follows:

\indent \quad a) If $G=1$,  equation \eqref{eq:no mass beyond 1/G}
tells that $\lim_{n\to \infty} F_n(T)=1$ for each $T>1$ and then
Theorem \ref{theor:JentzschSzego}, ii) says
that $\mu_n$ converges to $\Lambda$, and, so, $f \in \S$.

\indent \quad b) Equation \eqref{eq:some mass before 1/G} combined
with a diagonal argument implies that if $G<1$, a subsequence of the
$\rho_n$ converges to a probability measure whose essential support
reaches $1/G$. In particular, the  $\rho_n$ does not
converge to $\delta_1$ and $f \notin \S$.

\smallskip

We shall give a proof of  Theorem \ref{theor:gauge and convergence
of probabilities} in Section \ref{sect:proof of gauge and
convergence of probabilities}.

\section{Coefficients and zeros of polynomials}\label{sect:coefficients and zeros}

Next  we collect  a number of general results connecting
coefficients and zeros of polynomials. We also give proofs, based on
those connections and to be used later on, of Theorems~\ref{theor:JentzschSzego} and \ref{theor:Szego}.

For a polynomial $P$ and integer $n \ge \deg(P)$, we write $P$ as $P(z)=\sum_{k=0}^n b_k z^k$, where $b_k=0$ for $\deg(P) <k\le n$. We let $\Z_{(P,n)}$
denote its zero (multi-)set, maintaining the convention that if $\deg(P)<n$, then $P$ has a zero of
multiplicity  $n-\deg(P)$ at $\infty_{\scriptscriptstyle \C}$.

For each $t \ge 0$, we denote by $F_{(P,n)}(t)$ the proportion (with respect
to $n$) of the  zeros of~$P$ within the disk $\{w\in \C: |w|\le t\}$. Notice that
$\lim_{t\to \infty} F_{(P,n)}(t)=\deg(P)/n$.

We shall frequently appeal to the \textit{reversed companion  polynomial} $Q$ of $P$ \textit{with respect} to~$n$:
$$
Q_{(P,n)}(z)=z^n P(1/z)=\sum_{k=0}^n b_{n-k} z^k\, ;
$$
the zeros of $Q_{(P,n)}$ are the reciprocals of the zeros of $P$.

Furthermore, we order the zeros of $P$ according to their modulus and denote them by $w_1, w_2, \ldots, w_n$:
$$
|w_1|\le |w_2|\le \dots\le |w_n|\, ,
$$
keeping in mind that the last $n-\deg(P)$ of those are $=\infty_{\scriptscriptstyle \C}$.

\subsection{Jensen's formula} For $n \ge \deg(P)$, an application of Jensen's formula to
both $P$ and its reversed companion $Q_{(P,n)}$ inside the unit disk gives
that
$$
\sum_{w \in \Z_{(P,n)}} \big|\ln|w|\big|=\frac{1}{2\pi}\int_0^{2\pi} \ln \frac{|P(e^{\imath \vartheta})|^2}{|b_0||b_n|} d\vartheta\, .
$$
Now, for any $T>1$ we have
$$\begin{aligned}
\frac{1}{n}\sum_{w \in \Z_{(P,n)}} \big|\ln|w|\big|&\ge \frac{1}{n}\sum_{\substack{w \in \Z_{(P,n)};\\|w|>T}} \big|\ln|w|\big|+\frac{1}{n}\sum_{\substack{w \in \Z_{(P,n)};\\|w|<1/T}} \big|\ln|w|\big|\\[3pt]&\ge (\ln T) (1-F_{(P,n)}(T))+(\ln T) F_{(P,n)}(1/T)\, ,
\end{aligned}
$$
and, therefore,
\begin{equation}\label{eq:weaktype jensen}
\ln(T)\,  \big(1-F_{(P,n)}(T)+F_{(P,n)}(1/T)\big)  \le \frac{1}{2\pi}\int_0^{2\pi}\ln \frac{\sqrt[n]{|P(e^{\imath \vartheta})|^2}}{\sqrt[n]{|b_0|}\sqrt[n]{|b_n|}} d\vartheta\, , \quad \mbox{for all $T>1$}\, .
\end{equation}

\

This inequality \eqref{eq:weaktype jensen} readily gives a \textbf{proof of Theorem \ref{theor:JentzschSzego} {i)} and of Theorem \ref{theor:Szego}}. Consider $f \in \F$. Assume without loss of generality that $a_0=1$. Apply \eqref{eq:weaktype jensen} to the partial sum $s_n$ to get
\begin{equation}
\label{eq:weaktype jensen for sections}
\ln(T)\,   \big(1-F_n(T)+F_n(1/T)\big)  \le \frac{1}{2\pi}\int_0^{2\pi}\ln \frac{\sqrt[n]{|s_n(e^{\imath \vartheta})|^2}}{\sqrt[n]{|a_n|}} d\vartheta\, , \quad \mbox{for all $T>1$}\, .
\end{equation}

Since the radius of convergence  of $f$ is 1, one has that
$$
\limsup_{n\to \infty} \max_{|z|\le R} \sqrt[n]{|s_n(z)|}=R\, , \quad \mbox{for any $R \ge 1$}.
$$
Let $(n_k)_{k \ge 1}$ be any  increasing sequence such that $\lim_{k \to \infty} |a_{n_k}|^{1/n_k}=1$.  Since, for each $T >1$, one has $\lim_{n\to \infty} F_n(1/T)=0$, we obtain from \eqref{eq:weaktype jensen for sections} that
$$
\ln(T)\,   \limsup_{k \to \infty} (1-F_{n_k}(T)) \le 0\, , \quad \mbox{for any $T>1$}\, ,
$$
and so $\lim_{k \to \infty} F_{n_k}(T)=1$, for any $T>1$. This proves both  Theorem \ref{theor:JentzschSzego}\,i), and also  Theorem \ref{theor:Szego}.

Notice that actually the argument gives the following  general inequality:
\begin{equation}\label{eq:prior Carlson}
\ln(T)\,   \big(1-\liminf\limits_{n\to \infty}F_n(T)\big) \le -\ln
\big({\liminf\limits_{n\to \infty}\sqrt[n]{|a_n|}}\big) \, , \quad
\mbox{for any $T>1$}\, .
\end{equation}
Cf. \cite{Granville} and \cite{Hughes}.

\subsection{Coefficients as symmetric functions of the zeros}

For a polynomial $P$ and integer $n \ge \deg(P)$, if $b_0\neq 0$, the product of the zeros of $P$ and the coefficients of $P$ are related by
\begin{equation}\label{eq:product of zeros}
\frac{|b_0|}{|b_n|}=\prod_{w \in \Z_{(P,n)}} |w|\, .
\end{equation}

\smallskip

This identity readily gives (another) \textbf{proof of Theorem
\ref{theor:Szego}}. For $f \in \F$, assume without loss of generality that $f(0)=1$.
Fix $t <1$. From Hurwitz's theorem and the fact that $f(0)\neq 0$,
we obtain a constant $K_t>0$ (which depends on $t$ but not on $n$)
such that
$$
\prod_{w \in \Z_n: |w|\le t} |w|\ge K_t\, .
$$

For each $T>1$, after classifying the roots $w$ as $|w|\le t$,
$t<|w|\le T$, and $|w|>T$,  we may bound
$$
\frac{1}{|a_n|}\ge K_t \ t^{n(F_n(T)-F_n(t))} \ T^{n(1-F_n(T))}\ge K_t \ t^{n} \ T^{n(1-F_n(T))}\, .
$$
Taking  $n$-th roots and then  limits as $n \to \infty$, we conclude that if $\lim_{n \to \infty} \sqrt[n]{|a_n|}=1$ then
$$
 T^{\liminf\limits_{n\to \infty}F_n(T)}\ge t T\, .
$$
Since this is valid for any $t<1$ and since $T>1$, we deduce  that $\liminf_{n\to \infty}F_n(T)\ge 1$ and thus that $\lim_{n\to \infty}F_n(T)= 1$. This is Szeg\H{o}'s own argument in \cite{Szego}  to prove  Theorem \ref{theor:Szego}.

\smallskip

As we shall see, equation \eqref{eq:no mass beyond 1/G} of Theorem
\ref{theor:gauge and convergence of probabilities} will follow from
a variation of Szeg\H{o}'s argument but involving more coefficients
and not just $a_n$ and equation \eqref{eq:product of zeros}.

\smallskip

 From the expression of the coefficients of a polynomial $P$ as symmetric
 functions of its zeros (Vi\`ete's formulas) one  obtains, for $n \ge \deg(P)$, the inequality
\begin{equation}\label{eq:bounds1}
\frac{|b_{k}|}{|b_n|}\le \binom{n}{k}\prod_{j=k+1}^n|w_j|\, , \quad 0 \le k \le n\, ,
\end{equation}
with the convention that an empty product is 1. Upon considering the
reversed companion  polynomial $Q_{(P,n)}(z)$, one obtains the inequality
\begin{equation}\label{eq:bounds2}
\prod_{j=1}^k |w_j|\le \binom{n}{k}\frac{|b_{0}|}{|b_k|}\, , \quad 0 \le k \le n\, .
\end{equation}

\smallskip

To control the binomial coefficients appearing in \eqref{eq:bounds2}
we shall use the known elementary bound
\begin{equation}\label{eq:binomial bounds}
\binom{n}{k}\le e^{n \, H(k/n)}\,,  \quad \mbox{for} \ 0 \le k \le n \ \mbox{and} \ n \ge 1\, ,
\end{equation}
where  $H$ denotes the entropy function: $H(x)=-\big(x
\ln(x)+(1-x)\ln(1-x)\big)$ for $x \in [0,1]$. Notice that $H(x)=0$
if $x=0$ or $x=1$, and that $H$ decreases as $x$ goes from $1/2$ to
$1$.

\subsection{A proof of Szeg\H{o}'s Theorem \ref{theor:JentzschSzego}, ii).} What follows is a slight simplification of  Szeg\H{o}'s own argument for Theorem \ref{theor:JentzschSzego}, ii).

We assume with no loss of generality that $a_0=1$. Since $f(0)=a_0\neq 0$, we may fix $r>0$ and integer $N\ge 1$ such that no root of $s_n$ lies in the disk $\{|z|<r\}$.

For $z\in\C\setminus\{0\}$, we write $z/|z|=e^{\imath \theta(z)}$ with $\theta(z)\in[0,2\pi)$.

We shall prove that
\begin{equation}\label{eq:proof of szegos}
\lim_{k \to \infty} \frac{1}{n_k}\sum_{w \in \widetilde{\mathcal{Z}}_{n_k}} e^{-\imath m \theta(w)}=0\, ,\quad \mbox{for any integer $m \ge 1$}\, ,
\end{equation}
where $\widetilde{\mathcal{Z}}_n$ means $\mathcal{Z}_n$ with $\infty_\C$ excluded.
Since $\rho_{n_k} \to \delta_1$ as $k\to \infty$, the conclusion of Theorem \ref{theor:JentzschSzego}, ii),  will follow from \eqref{eq:proof of szegos} combined with  Weierstrass approximation theorem.

\smallskip

To prove \eqref{eq:proof of szegos}, fix an integer $m \ge 1$.

Let $\sigma_n$ denote the reversed companion polynomial $
\sigma_n(z)=\sum_{k=0}^n a_{n-k} z^k
$ of the partial sum $s_n$ with respect to $n$.
An application of Newton's identities to $\sigma_n$ gives  that
$$
\sum_{w \in \mathcal{Z}_n}\frac{1}{w^m}=\sum_{w \in \widetilde{\mathcal{Z}}_n}\frac{1}{w^m}=\Psi_m(a_1, a_2, \ldots, a_m)\, ,
$$
where $\Psi_m$ is a certain function defined in $\C^m$. Consequently,
$$
\lim_{n \to \infty} \frac{1}{n}\sum_{w \in \mathcal{Z}_n}\frac{1}{w^m}=0\,  .
$$

\smallskip

Now, for $n \ge N$, write
$$
\frac{1}{n} \sum_{w \in \widetilde{\mathcal{Z}}_n}e^{-\imath m \theta(w)}=\frac{1}{n} \sum_{w \in \mathcal{Z}_n}\frac{1}{w^m}+\frac{1}{n}\sum_{w \in \widetilde{\mathcal{Z}}_n} e^{-\imath m \theta(w)}(1-|w|^{-m})\, .
$$
For $T>1$, we may bound the last sum in the expression above as
$$\begin{aligned}
\Big|\frac{1}{n}\sum_{w \in \widetilde{\mathcal{Z}}_n} e^{-\imath m \theta(w)}(1-|w|^{-m})\Big|&\le \frac{1}{n}\sum_{w \in \widetilde{\mathcal{Z}}_n} |1-|w|^{-m}|\\&\le
F_n(1/T) (r^{-m}-1)+(T^m-1)\\&+ (1-F_n(T))+(1-T^{-m})\, .
\end{aligned}
$$
Since  $\lim_{n \to \infty} F_n(1/T)=0$ and, by hypothesis, $\lim_{k \to \infty} (1- F_{n_k}(T))=0$, we conclude that
$$
\limsup_{k \to \infty} \Big|\frac{1}{n_k} \sum_{w \in \widetilde{\mathcal{Z}}_{n_k}}e^{-\imath m \theta(w)}\Big|\le T^m-T^{-m}\, , \quad \mbox{for any $T>1$}\,,
$$
which gives \eqref{eq:proof of szegos}.

\subsection{Cauchy's and Van Vleck's bounds}

These are classical bounds for the location of zeros of a polynomial
$P$ in terms of (all or some of) its coefficients. Consult
\cite{Marden}, chapters VII and VIII, or \cite{Henrici}, chapter 6,
and also the original paper \cite{VanVleck} of Van Vleck.

The bound of Cauchy asserts that all the zeros of the polynomial $P$ lie
in $\{w\in \C: |w|\le C_P\}$, where $C_P$ is the unique  positive
root of
$$
|b_n|x^n=\sum_{k=0}^{n-1} |b_k| x^k\, .
$$
We understand that if $b_n=0$,  then $C_P=+\infty$.


\smallskip

Upon considering the reversed companion polynomial $Q_{(P,n)}$ one observes
that all the roots of $P$ lie  in  $\{w\in \C: |w|\ge c_P\}$ where
$c_P$ is the unique positive root of the equation in~$y$:
$$
|b_0|=\sum_{k=1}^{n} |b_{k}| y^k\, .
$$
Notice that
$$
|b_0|\le n \Big(\max_{1\le k\le n} |b_k|\Big) \max(1, c_P)^n\,.
$$

\smallskip

\textbf{Van Vleck's bounds} assert that for  $1 \le m \le n$ at
least $m$ zeros  of $P$ lie in the disk $\{w: |w|\le V^m_P\}$,  where
$V^m_P$ is the unique  positive root of
$$
|b_n|x^n=\sum_{j=0}^{m-1} \binom{n-j-1}{m-j-1} |b_j|x^j\, .
$$
The case $m=n$ is Cauchy's bound: $V^n_P=C_P$.
Again, we understand that $V^m_P=+\infty$ if $b_n=0$.

Upon applying these bounds to the reversed companion  polynomial
$Q_{(P,n)}$ with respect to $n$ we deduce that,  if $1 \le m \le n$, the polynomial $P$
has $m$ roots in $\{w\in \C: |w| \ge v^m_P\}$,  where $v^m_P$ is the
unique positive root of the equation
$$
|b_0|=\sum_{k=n-m+1}^n \binom{k-1}{k-(n-m)-1}|b_k|y^k\, .
$$

Using that
$$
\sum_{k=n-m+1}^n \binom{k-1}{k-(n-m)-1}=\binom{n}{m-1}\, ,
$$
we deduce that
\begin{equation}\label{eq:vanvleck bound}
|b_0|\le \binom{n}{m-1} \Big(\max_{n-m+1\le k\le n}|b_k|\Big)\, \max(1, v^m_P)^n\, .
\end{equation}

\medskip

\section{Proof of theorem \ref{theor:gauge and convergence of probabilities}}\label{sect:proof of gauge and convergence of probabilities}

Let $f\in \F$. We assume with no loss of generality that $a_0=1$.

\smallskip

First we deal with the \textbf{proof of inequality  \eqref{eq:no mass
beyond 1/G}}.

For each $n \ge 1$, present the zeros of $s_n$ in ascending order of
modulus as $|w^{(n)}_1|\le \dots \le |w^{(n)}_n|$. Recall that if
the degree of $s_n$ is $m\le n$, we, conveniently and
conventionally, understand that the last $n-m$ of these zeros are
$\infty_{\scriptscriptstyle \C}$.

The bounds \eqref{eq:bounds2} translate into
$$
\prod_{j=1}^k |w^{(n)}_j|\le \binom{n}{k}\frac{1}{|a_k|}\, , \quad \mbox{for} \ 0 \le k \le n \ \mbox{and}\  n \ge 1\, .
$$

\

Fix $T>1$ and $\gamma \le 1/2$. Fix also $t<1$, which later on will
tend to 1. From Szeg\H{o}'s argument of Section
\ref{sect:coefficients and zeros}, maintaining the notation therein,
we obtain that
$$
K_t \, t^n \, T^{k-nF_n(T)}\le \binom{n}{k} \frac{1}{|a_k|}\, , \quad \mbox{for} \ 0 \le k \le n\ \mbox{and}\  n \ge 1\, .
$$
If we restrict $k$ to the range $(1-\gamma) \, n\le k \le n$ we deduce,
using the bound \eqref{eq:binomial bounds}, that
$$
K_t\, t^n\, T^{n(1-\gamma-F_n(T))}\le e^{n
H(1-\gamma)}\frac{1}{|a_k|}\, , \quad \mbox{for $(1-\gamma) \, n\le k \le n$ and $n \ge 1$}\, .
$$
and, then, that
$$
K_t\, t^n\, T^{n(1-\gamma-F_n(T))}\le e^{n
H(1-\gamma)}\frac{1}{A_n(\gamma)}\, , \quad \mbox{$n \ge 1$}\, .
$$
Upon extracting $n$-th roots, letting first $n \to \infty$, and then
letting $t \uparrow 1$, we deduce
$$
T^{(1-\gamma-\liminf\limits_{n \to \infty}F_n(T))}\le e^{H(1-\gamma)}\frac{1}{L(\gamma)}\, .
$$
Letting $\gamma \downarrow 0$, and using  that $H(1)=0$, we deduce  that
$$
T^{1-\liminf\limits_{n \to \infty}F_n(T)}\le \frac{1}{G}\, ,
$$
or, as claimed,
$$
\liminf\limits_{n \to \infty}F_n(T)\ge 1- \frac{\ln(1/G)}{\ln(T)}\, .
$$
(Compare this last inequality with  inequality \eqref{eq:prior
Carlson}.)

\smallskip

Next, we turn to the \textbf{verification of inequality \eqref{eq:some
mass before 1/G}}.

We assume that $G<1$, since otherwise the result is trivially true,
and we let $1<T<1/G$. Fix $\varepsilon >0$ so that
$(G+\varepsilon)T<1$. Since $H(0)=0$, we may choose, and fix,  $\gamma\in
(0,1/2)$ so that
$$
L(\gamma) \, e^{H(\gamma)}\le G+\varepsilon/2\, .
$$
For an infinite subset $\mathcal{N}$ of $\N$  one has that
\begin{equation}
\label{eq:bound G plus eps}
  \sqrt[n]{A_n(\gamma)} \ e^{H(\gamma)} \le
G+\varepsilon\, , \quad \mbox{for $n \in \mathcal{N}$}\, .
\end{equation}

\

For $n \in \mathcal{N}$, let $m_n=\lfloor \gamma n \rfloor +1$.
The Van Vleck's bounds, equation \eqref{eq:vanvleck bound}, applied to
$s_n$ gives that $s_n$ has at least $m_n$ roots with modulus no less
than $v_n$, where $v_n$ satisfies
\begin{equation}
\label{eq:lower bound An} 1\le \binom{n}{\lfloor \gamma n \rfloor} \
A_n(\gamma) \max(1, v_n)^n\, , \quad \mbox{for $n \in
\mathcal{N}$}\, .
\end{equation}
Observe that
$$
F_n(v_n)\le \frac{n-m_n}{n}\le 1-\gamma\, ,  \quad \mbox{for $n \in
\mathcal{N}$}\,.
$$

From the bound  \eqref{eq:binomial bounds} and inequality
\eqref{eq:bound G plus eps} above, we deduce from  inequality \eqref{eq:lower bound An} that
$$
1\le (G+\varepsilon) \max(1, v_n)\, , \quad \mbox{for $n \in \mathcal{N}$}\, .
$$
Since $G+\varepsilon<1$, this means that $v_n >1$ and, in fact, that
$$
\frac{1}{G+\varepsilon}<v_n\, , \quad \mbox{for $n \in \mathcal{N}$}\, .
$$
Therefore
$$
F_n(T)\le 1-\gamma\, , \quad \mbox{for $n \in \mathcal{N}$}\, ,
$$
and consequently $$\liminf_{n \to \infty} F_n(T)\le 1-\gamma <1\, .$$
This completes the proof of Theorem \ref{theor:gauge and convergence
of probabilities}. \hfill $\Box$

\

It should be mentioned that the proof above of equation \eqref{eq:no
mass beyond 1/G} of Theorem \ref{theor:gauge and convergence of
probabilities} is a direct adaptation of Szeg\H{o}'s own  argument
in \cite{Szego} to prove his Theorem \ref{theor:Szego}, which we
have discussed in Section \ref{sect:coefficients and zeros}, while
the proof of \eqref{eq:some mass before 1/G} of Theorem
\ref{theor:gauge and convergence of probabilities} is  a refinement
of a suggestion of P. Tur\'an which appears as a \textit{note added in
proof} in the paper \cite{ErdosFried} of Erd\H{o}s and Fried.

\begin{remark}\label{rem:infinite Ostrowsky gaps}\rm
The proof above of Theorem \ref{theor:gauge and convergence of
probabilities} actually gives that if $\gamma \in (0,1)$  then
$$
\liminf_{n \to \infty} F_n(T)\le 1-\gamma\, , \quad \mbox{for any $T<\big(L(\gamma)e^{H(\gamma)}\big)^{-1}$}\, .
$$
In the argument above we have just  used the case $\gamma$ close to $0$, but if we let $\gamma \uparrow 1$ we obtain:
$$
\liminf_{n \to \infty} F_n(T)=0\, , \quad \mbox{for any $T<\big(\lim_{\gamma \uparrow 1}L(\gamma)\big)^{-1}$}\, .
$$
Of course, this is only relevant if $\lim_{\gamma \uparrow 1}L(\gamma)<1$. Power series for which this occurs, like   $\sum_{n=0}^\infty z^{n!}$, are said to have \textit{infinite} Ostrowsky gaps, see {\upshape \cite{ErdosFried}} Theorem II.
\end{remark}

\section{Random power series}\label{sect:random power series}

Next we turn to random power series. Our aim is to analyze, using
the deterministic machinery of the previous sections, the
probability that such a random power series is a Szeg\H{o} power
series. We point out to \cite{Barucha} and \cite{Kahane} as general
references on random polynomials and on random power series.

\subsection{The iid case} To start with, let $X$ be any non null complex valued random variable. Consider a sequence $(X_n)_{n\ge 0}$ of completely independent clones of $X$ in a certain probability space $(\Omega, \P)$. For each $\omega\in \Omega$, let $f_\omega$ denote the power series
$$
f_\omega(z)=\sum_{k=0}^\infty X_k(\omega) z^k\, .
$$
This model of random power series is usually called \textit{Kac ensemble},
particularly so if $X$ is a gaussian variable.

The radius of convergence of $f_\omega$ is a random variable, but it
turns out to be almost surely constant;  actually, the
Borel--Cantelli lemma gives directly the following well-known
dichotomy.
\begin{lemma}\label{lemma:random radius}\mbox{}

\smallskip

\indent If\/ $\E(\ln^+|X|)<+\infty$, then $\limsup_{n \to \infty} \sqrt[n]{|X_n|}=1$, almost surely.
\\[6pt]
\indent    If\/ $\E(\ln^+|X|)=+\infty$, then $\limsup_{n \to \infty} \sqrt[n]{|X_n|}=+\infty$, almost surely.
\end{lemma}

In terms of the power series $f_\omega$,
this lemma means that if $\E(\ln^+|X|)<+\infty$,
then the radius of convergence of $f_\omega$ is almost surely 1,
while if $\E(\ln^+|X|)=\infty$, the radius of convergence of $f_\omega$ is almost surely 0.
In other terms,
under the condition $\E(\ln^+|X|)<+\infty$, the random power series $f_\omega$ is almost surely in $\F$.

\smallskip

Notice that if $\E(|\ln|X||)<+\infty$ then
almost surely $\lim_{n \to \infty} \sqrt[n]{|X_n|}=1$, and conversely.
Thus, if $\E(|\ln|X||)<+\infty$, condition \eqref{eq:szegos condition} of
Theorem \ref{theor:Szego} holds almost surely.

\smallskip

We include a proof of lemma \ref{lemma:random radius}, since later
on we shall adapt it to the case of non identically distributed random
coefficients.

\begin{proof}
If $\E(\ln^+|X|)<+\infty$ then $\sum_{n=0}^\infty \P(|X| \ge
e^{\alpha n})< +\infty$, for all $\alpha >0$. Since the~$X_n$ are
identically distributed this, in turn, is equivalent to
$\sum_{n=0}^\infty \P(|X_n| \ge e^{\alpha n})< +\infty$. The lemma of Borel--Cantelli (no independence needed) gives then, for each $\alpha
>0$, that $\limsup_{n \to \infty}\sqrt[n]{|X_n|}\le e^{\alpha}$ almost surely, and
consequently, $\limsup_{n \to \infty}\sqrt[n]{|X_n|}\le 1$ almost
surely.

Since $X$ is non null, for some $\delta>0$ we have  that $\P(|X|\ge
\delta)>0$. Since the $X_n$ are identically distributed this implies
that $\sum_{n=0}^\infty \P(|X_n| \ge \delta)= +\infty$. Now, using
independence, the  lemma of Borel--Cantelli  gives then that
$\limsup\sqrt[n]{|X_n|}\ge 1$, almost surely.

\smallskip

If $\E(\ln^+|X|)=+\infty$, then  $\sum_{n=0}^\infty \P(|X| \ge e^{\alpha n})= +\infty$, for all $\alpha >0$. Now, independence and the lemma of Borel--Cantelli gives that $\limsup\sqrt[n]{|X_n|}\ge e^{\alpha}$ for all $\alpha>0$, and so $\limsup\sqrt[n]{|X_n|}=+\infty$ almost surely.
\end{proof}

Fix a non null random variable $X$ and, as above, let  $(X_n)_{n \ge
0}$ be a sequence of completely independent clones of $X$. For
$\gamma \in (0,1)$ and $n \ge 0$, define the random variable $$
A_n(\gamma)=\max_{(1-\gamma) n \le k \le n} |X_k|\, .
$$

\begin{lemma}\label{lemma:non null implies gauge=1} If $X$ is a non null random variable, then for each $\gamma \in (0,1)$
$$
\liminf_{n \to \infty} \sqrt[n]{A_n(\gamma)}\ge 1 \quad \mbox{almost surely}\, .
$$
\end{lemma}
Notice that in Lemma \ref{lemma:non null implies gauge=1} no
assumption on $\E(\ln^+|X|)$ is required; just the trivial
assumption that $X$ is non null (and independence of the clones, of
course) implies that almost surely the sequence $(X_n(\omega))_{n
\ge 1}$ can not be too small for long stretches of $n$.

Observe also that Lemma \ref{lemma:non null implies gauge=1} does
not hold for $\gamma=0$: simply take $X$ to be a Bernoulli random
variable.

\begin{proof}
Fix $\gamma \in (0,1)$. We have to verify  that
$$
\P\big(\liminf_{n \to \infty} \sqrt[n]{A_n(\gamma)}\ge 1\big)=1\, ,
$$
or, equivalently, that, for each  $\varepsilon >0$:
$$
\P\big(\sqrt[n]{A_n(\gamma)}\le (1-\varepsilon), \, \, \mbox{infinitely many $n$}\big)=0\, ,
$$
which, in turn, by  the lemma of Borel--Cantelli (no independence
assumption needed) reduces to prove that
$$
\sum_{n=0}^\infty\P\big(\sqrt[n]{A_n(\gamma)}\le (1-\varepsilon)\big)<+\infty\, .
$$
Since each $X_n$ is a clone of $X$, all we have to show is that
$$
\sum_{n=0}^\infty\P\big(|X|\le  (1-\varepsilon)^n\big)^{\gamma n}<+\infty\, .
$$

\smallskip

Since $X$ is non null, there is $\delta >0$ such that
$\P(|X|< \delta)<1$. Now for $n \ge N=N(\delta, \varepsilon)$ one has that $(1-\varepsilon)^n < \delta$, and, consequently,
\begin{equation*}
\sum_{n\ge N}^\infty\P\big(|X|\le  (1-\varepsilon)^n\big)^{\gamma n}\le
\sum_{n\ge N}^\infty\P\big(|X|< \delta\big) ^{\gamma n }<+\infty
\, .\qedhere
\end{equation*}
\end{proof}

 \begin{theorem}
 \label{theor:Gamma 1 as}
If $X$ is a {\upshape(}non-null{\upshape)} random variable and
$\E(\ln^+|X|)<+\infty$, then almost surely the gauge $G$ of
$f_\omega$ is $1$.
 \end{theorem}
\begin{proof}
The assumption $\E(\ln^+|X|)<+\infty$, implies,
by lemma \ref{lemma:random radius}, that
almost surely the radius of convergence of $f_\omega$ is 1.
And then, the hypothesis that $X$ is non null implies, by
Lemma \ref{lemma:non null implies gauge=1}, that almost surely $f_\omega$ has gauge 1.
\end{proof}

As a consequence of Theorem
\ref{theor:Gamma 1 as} and Theorem \ref{theor:CarlsonBourion} we
obtain the following theorem of Ibragimov and Zaporozhets, \cite{IbraZapo}. Consult also
\cite{Arnold}, \cite{Hughes} and \cite{ShparoShur}.
\begin{theorem}[Ibragimov--Zaporozhets]\label{theor:IZ}
For any {\upshape (}non null{\upshape )} random variable $X$ satisfying
$\E(\ln^+|X|)<+\infty$, the sequence $(\mu_n)_{n \ge 0}$ of
random probabilities converges  almost surely to the uniform probability $\Lambda$ on $\partial \D$; in other terms,
almost all power series $f_\omega$ are Szeg\H{o} power series.
\end{theorem}


\begin{remark}\rm
Under the (stronger) hypothesis  $\E\big(\big|\ln|X|\big|\big)<+\infty$, Szeg\H{o}'s condition {\upshape \eqref{eq:szegos condition}} is almost surely satisfied and, in this case, one obtains the conclusion of Theorem  {\upshape \ref{theor:IZ}} directly from Theorem {\upshape \ref{theor:Szego}}, and there is no need   to appeal to Theorem {\upshape \ref{theor:CarlsonBourion}}. See also {\upshape \cite{Arnold}}.
\end{remark}

\begin{remark}\rm If $X$ is a Bernoulli variable with $\P(X=1)=p \in (0,1)$, then almost all $f_w$ belong to $\S$, but almost none of the $f_\omega$
satisfy the condition \eqref{eq:szegos condition} of Theorem \ref{theor:Szego}.
\end{remark}

\noindent \textbf{Expected distribution function in the iid case.}
For each $\omega \in \Omega$, denote by $\mu_{n,\omega}$ and
$\rho_{n,\omega}$ the probability measures associated to $f_\omega$
and let $F_{n, \omega}(t), t \ge 0$, denote the distribution
function of  $\rho_{n,\omega}$.

Consider the expected  distribution function
$$\Phi_n(t)=\int_\Omega F_{n, \omega}(t)\, d\P(\omega)\, , \quad \mbox{for $t \ge 0$}\, .
$$

Since the $X_j$ are completely independent and identically
distributed, the section $s_{n, \omega}(z)=\sum_{k=0}^n X_k(\omega)
z^k$ and its reversed companion $\sum_{k=0}^n X_{n-k}(\omega) z^k$
are identically distributed and, consequently, the following
symmetry holds:
$$
\Phi_n(t)=1-\Phi_n(1/t)\, , \quad \mbox{for any  $0<t\le1$}\, .
$$
Notice that $\Phi_n(1)=1/2$, for each $n \ge 1$.

Recall that, by Hurwitz's theorem, $\lim_{n \to
\infty}F_{n, \omega}(t)=0$, for each $t <1$, almost surely, and so, by
dominated convergence, $\lim_{n \to \infty} \Phi_n(t)=0$ for each $t
<1$. Consequently, $\lim_{n \to \infty} \Phi_n(T)=1$, for each
$T>1$. The last convergence statement follows also from the fact the $G=1$
almost surely and from Theorems \ref{theor:CarlsonBourion} and
\ref{theor:gauge and convergence of probabilities}. Therefore,
$$
\lim_{n \to \infty} \Phi_n(t)=\begin{cases} 0, & t<1\,,\\
1/2,& t=1\, ,\\
1,& t>1\, .
\end{cases}
$$

For Gaussian  $X$ or, more generally, for  $X$
in the domain of attraction of a stable law of exponent $\alpha \in
(0,2]$, there are precise expressions for $\E(\mu_n(B))$ for any
Borel set $B\subset \C$; see~\cite{SheppVanderbei} and~\cite{IbraZeit}.

\subsection{Independent (not necessarily equidistributed) coefficients}

Let us consider now a sequence $(X_n)_{n \ge 0}$ of mutually
independent random variables in a certain probability space
$(\Omega, \P)$; no assumption now on a common distribution. As
above, we let
$$ f_\omega(z)=\sum_{n=0}^\infty X_n(\omega)\, z^n\, .
$$
For $n \ge 0$ and $\gamma \in (0,1)$, we denote
$A_n(\gamma)=\max\limits_{(1-\gamma)n \le k \le n}|X_k|$.

\smallskip

After  reviewing the discussion above of the iid case, it is easy to
come out with natural and simple conditions on the sequence of
independent variables $(X_n)_{n \ge 0}$ which are enough to
guarantee the conclusions of Lemmas \ref{lemma:random radius} and
\ref{lemma:non null implies gauge=1}.

\medskip

A) If for some $\varepsilon >0$, one has
\begin{equation}
\label{eq:condition near infty independent}
\sup_{n \ge 0} \E\big((\ln^+|X_n|)^{1+\varepsilon}\big)<+\infty\,,
\end{equation}
then for $\alpha >0$  and $n \ge 0$, Markov's inequality gives us that
$$
\P(|X_n|\ge e^{\alpha n})\le
\frac{\E((\ln^+|X_n|)^{1+\varepsilon})}{\alpha^{1+\varepsilon}n^{1+\varepsilon}}\, .$$
Therefore $$\sum_{n=0}^\infty \P(|X_n|\ge e^{\alpha n})<+\infty\, , \quad \mbox{for each $\alpha >0$}, ,$$ and, consequently, see the
proof of Lemma \ref{lemma:random radius}), we conclude that
$$\limsup_{n \to \infty} \sqrt[n]{|X_n|}\le 1\quad \ \mbox{almost surely}\,.$$

\smallskip

B)  If for some $\delta >0$, one has that
\begin{equation}
\label{eq:condition near 0 independent}
\inf_{n \ge 0}\P(|X_n| \ge \delta) >0\, ,\end{equation}
then (see the proof of Lemma \ref{lemma:random radius})  $$\limsup_{n \to \infty} \sqrt[n]{|X_n|}\ge 1\quad \mbox{almost surely}\, ,$$
and, besides, the proof of Lemma \ref{lemma:non null implies gauge=1} carries over and gives that
$$
\liminf_{n \to \infty} \sqrt[n]{A_n(\gamma)}\ge 1 \quad \mbox{almost surely}\, .
$$

\

Therefore we have:
\begin{theorem}\label{theor:szego for no id}
Under conditions {\upshape \eqref{eq:condition near infty
independent}} and {\upshape \eqref{eq:condition near 0 independent}}
above, the random power series $f_\omega$ is almost surely a
Szeg\H{o} power series.
\end{theorem}

Conditions analogous to \eqref{eq:condition near infty independent}
and \eqref{eq:condition near 0 independent} appear also in
\cite{Pritsker} to obtain a result like Theorem~\ref{theor:szego
for no id}.

\begin{remark}\rm
Conditions \eqref{eq:condition near infty independent} and
\eqref{eq:condition near 0 independent} are not as demanding than
those appearing in \cite{ShmerlingHochberg}: no continuous densities
or finite moments assumptions other than the log-moment above.

Under the assumptions of \cite{ShmerlingHochberg}, Theorem 1,  one actually
has $\lim_{n \to \infty} \sqrt[n]{|X_n|}=1$ almost surely,
and thus almost surely Theorem \ref{theor:Szego} applies
(no need to appeal to Theorem \ref{theor:gauge and convergence of
probabilities}), and the sequence $(\mu_n)_{n \ge 0}$ of random
probabilities converges  almost surely to the uniform probability
$\Lambda$  on $\partial\D$.
\end{remark}

\noindent \textbf{Bernoulli trials.} Let  $(X_n)_{n \ge 0}$ be a
sequence of completely independent Bernoulli variables with
$p_n=\P(X_n=1)$, for $n \ge 0$. Notice that  condition
\eqref{eq:condition near infty independent} is trivially satisfied in
this case.

\smallskip

Because of independence and the Borel--Cantelli lemmas, for the
radius of convergence  we have in this case that:

\smallskip
a) \,  if $\sum_{n=1}^\infty p_n=+\infty$, then almost surely the radius
of convergence of $f_\omega$ is 1;

\medskip

\quad b) \,  if $\sum_{n=1}^\infty p_n<+\infty$,
then $f_\omega$ is almost surely a polynomial and its radius of
convergence is $+\infty$.

As for belonging to $\S$, we have that if $\inf_{n \ge 0}
p_n>0$, then both conditions, \eqref{eq:condition near infty
independent} and \eqref{eq:condition near 0 independent}, are
satisfied and almost surely $f_\omega$ is in $\F$ and also in $\S$.

\medskip

Consider now the case where
$$(\dag) \quad p_n=1/n\,, \mbox{for $n \ge 1$}\, .$$
In this case $f_\omega$ has radius of convergence 1, almost surely.
Condition \eqref{eq:condition near 0 independent} is not satisfied
and, in fact, as we shall presently verify,  the index of $f_\omega$
is almost surely 1.

Fix $\gamma \in (0,1)$ and let $(n_k)_{k \ge 1}$ such that
$(1-\gamma)>n_{k-1}/n_{k}\to (1-\gamma)$.

Then $\P\big(A_{n_k}(\gamma)=0 \big)=\prod_{(1-\gamma)n_k \le j \le
n_k} (1-1/j)$ and so, for $k$ large enough,
$$
\P\big(A_{n_k}(\gamma) =0\big)\ge \exp\Big(-2 \sum_{(1-\gamma)n_k
\le j \le n_k}1/j\Big) \ge C (1-\gamma)^2\, .
$$
Since $(1-\gamma)>n_{k-1}/n_{k}$, the events $\{A_{n_k}(\gamma)\}$ are independent. The lemma of Borel--Cantelli now gives that
$$
\P\big(A_{n_k}(\gamma)=0\quad  \mbox{\rm infinitely often}\big)=1\, ,
$$
and, consequently,
$$
\P(L(\gamma)=0)=1\, , \quad \mbox{for each $\gamma \in (0,1)$}\, .
$$
Therefore the index $\Gamma$ is almost surely 1. In this case we are, almost surely, in the situation of Remark \ref{rem:infinite Ostrowsky gaps}.

So, for probabilities $p_n$ satisfying $(\dag)$ the random power
series $f_\omega$ has almost surely radius of convergence 1, but
almost surely $f_\omega$ is not a Szeg\H{o} power series.

\medskip

\section{Limits of zero counting measures}

It is natural to ask what are the possible (weak) limits of the sequence of probabilities $\rho_n$ associated to a given $f\in \F$. Let us denote by $\mathcal{L}_f$ the collection of those limits points; the elements of $\mathcal{L}_f$ are probability measures on $[0,+\infty]$.

If the index $\Gamma$  of $f$ is 0, then, by Theorem
\ref{theor:CarlsonBourion}, the only such limit is $\delta_1$, i.e.
$\mathcal{L}_f=\{\delta_1\}$, and, conversely.

\subsection{Power series with coefficients 0 or 1}

If all the coefficients of the power series $f\in \F$ are 0 or 1 then
we have the following  complete description of $\mathcal{L}_f$:

 \begin{proposition}
If $f\in \F$ has index $\Gamma$, then
$$\mathcal{L}_f=\{(1-u)\delta_1+u \delta_{\infty_{\scriptscriptstyle \C}}: 0\le
u \le \Gamma\}\, .$$
 \end{proposition}
\begin{proof}
Denote by $\mathcal{M}$ the collection of indexes $n$ such that $a_n=1$. With no loss of generality we assume that $a_0=1$.

For $T >1$ and $n \in \mathcal{M}$, equation
\eqref{eq:weaktype jensen for sections} gives us
$$
\ln(T)\, \big(1-F_n(T)\big)\le \frac{1}{2\pi}\int_0^{2\pi}\ln
\sqrt[n]{|s_n(e^{\imath \vartheta})|^2}\,
d\vartheta\le \frac{2}{n}\ln(n+1)\, ,
$$
and so, for any $T>1$,
$$
\lim_{\substack{n \in \mathcal{M}; \\n \to \infty}} F_n(T)=1\, .
$$
Consequently, $\rho_n$ tends to $\delta_1$ as $n \in \mathcal{M}$ tends to $\infty$.

For integer $n\ge 0$, let $m(n)=\max\{m \in \mathcal{M}: m \le n\}$. Observe that
$$
F_n\equiv \frac{m(n)}{n} F_{m(n)}\, , \quad \mbox{for $n \ge 0$}\, .
$$
Notice also that each $\rho_n$ has mass $1-m(n)/n$ at $+\infty$.

Thus, for an increasing sequence $(n_k)_{k \ge 1}$ of indexes,  the sequence $\rho_{n_k}$ has limit, say,  $\rho$ if and only if the sequence $m(n_k)/n_k$ converges, say, to $\alpha$; in that case $\rho=\alpha \delta_1+(1-\alpha)\delta_{\infty_{\scriptscriptstyle \C}}$.

Since the possible limits of sequences $m(n_k)/n_k$ cover exactly the interval $[1-\Gamma, 1]$, the result follows.
\end{proof}

The simple argument above is modeled upon part of the discussion of \cite{BlochPolya}.

\subsection{A universal power series}

Let $\mathcal{P}$ be  the set of probability (Borel) measures in $[0,+\infty)$ and let $\mathcal{P}_1$ be the subset of $\mathcal{P}$ of those probabilities supported in $[1,+\infty)$. We endow $\mathcal{P}$ with the L\'evy--Prokhorov metric (distance) $D$ with respect to Euclidean distance in $[0,\infty)$; convergence with respect to this metric $D$ and weak convergence coincide.

In this section we shall exhibit an example of a
\textit{single} power series $f(z)=\sum_{n=0}^\infty a_n z^n\in \F$ such that \textit{every} probability
measure in $[1,+\infty)$ is a limit of a subsequence of the probability measures
$(\rho_n(f))_{n \ge 0}$ associated to the sequence of sections of $f$. The power series $f$ is \textit{universal} in the sense
that the probabilities measures $\rho_n(f)$  are dense
in $\mathcal{P}_1$:
$$
\mbox{\rm clos}_D\{\rho_n(f): n \ge 1\}=\mathcal{P}_1\, .
$$

\smallskip

 The countable collection $\mathcal{D}\subset \mathcal{P}_1$ of probabilities of the form
$(1/m)\sum_{j=1}^m \delta_{q_j}$, where $m$ is an integer $m\ge 1$ and $1 < q_1<\cdots<q_m$ are rational numbers,
 is dense in $\mathcal{P}_1$. Let
$(\varphi^{(k)})_{k \ge 1}$ be a sequence of probabilities which
contains each of the probabilities in $\mathcal{D}$  infinitely many times.

\smallskip

The power series $f\in \F$ will have the form
$$
f(z)=1+\sum_{j=1}^\infty z^{N_j} Q_j(z)\, .
$$
The $Q_j$ are polynomials with $Q_j(0)=1$. Denote $P_k(z)=1+\sum_{j=1}^k z^{N_j} Q_j(z)$ and $d_k=\mbox{\rm deg}(P_k)$. The integers $N_j$ grow so fast that $N_k > d_{k{-}1}$ for any $k \ge 1$. Thus  $s_{d_k}(f)=P_k$, for $k \ge 1$.

The polynomials $Q_j$ and the integers $N_j$ will be defined iteratively so that
$$
D\big(\rho_{d_k}(f), \varphi^{(k)}\big)\le \frac{1}{k}\, \quad \mbox{for any $k \ge 1$}\, .
$$

Before starting the actual construction  we record a few preliminary lemmas. We shall need the following estimate of the distance $D$ of two specific probabilities whose verification follows directly form the definition of $D$.

\begin{lemma}\label{lemma:distance of some probabilities}
Let $\mu \in \mathcal{P}_1$ be given by $\mu=(1/m)\sum_{j=1}^m \delta_{r_j}$ where $1< r_1<\cdots <r_m$.

Let $\varepsilon >0$ be such that the intervals $I_j(\varepsilon):=(r_j-\varepsilon, r_j+\varepsilon)$ are pairwise disjoint.

Let $\nu\in \mathcal{P}$ be given by
$$
\frac{1}{mk+h} \Big(\sum_{j=1}^m \sum_{l=1}^k \delta_{s_{j,l}}+\sum_{i=1}^h \delta_{t_i}\Big)
$$
where $s_{j,l}\in I_j(\varepsilon)$, for $1\le j \le m, 1\le l \le k$ and $t_i \ge 0$ for $1\le i \le h$.

If $\dfrac{h}{m\,k} < \varepsilon $, then $D(\mu,\nu) < \varepsilon$.
\end{lemma}

For integer $M\ge 1$, we denote by $\mathcal{U}_M$ the collection of the $M$-th roots of unity.

\begin{lemma}
\label{lemma:near zeros 1-zN}
For  integer $M \ge 1$ and radius $r>0$, one has
$$\Big|1-\Big(\frac{z}{r}\Big)^M\Big|\ge 3-e\, , \quad \mbox{for any $z$ such that $\mbox{\rm{dist}}(z,r\mathcal{U}_M)\ge r/M$.}
$$
\end{lemma}
\begin{proof}
Let $z$ be such that $|z-ru|=r/M$, where $u \in \mathcal{U}_M$. Write  $z=r(u+w/M)$, with $|w|=1$. We have that
$$\begin{aligned}
\Big|1-\Big(\frac{z}{r}\Big)^M\Big|&=\Big|1-\Big(u+\frac{w}{M}\Big)^M\Big|=\Big|\sum_{j=1}^M \binom{M}{j}\frac{w^j u^{M-j}}{M^j}\Big|\ge 1 -\Big|\sum_{j=2}^M \binom{M}{j}\frac{w^j u^{M-j}}{M^j}\Big|\\
&\ge 1 -\sum_{j=2}^M
\binom{M}{j}\frac{1}{M^j}=1-\bigg(\Big(1+\frac{1}{M}\Big)^M-2\bigg)\ge
3-e\,.
\end{aligned}
$$
Let now $\Omega$ be the domain $\Omega=\big\{z\in \C: \, \mbox{\rm
dist}\big(z, r\, \mathcal{U}_M\big) \ge ({r}/{M})\big\}$, and let
$g$ be the function
$$
g(z)=\dfrac{1}{1-\big({z}/{r}\big)^M}\, .
$$
The function $g$ is holomorphic and does not vanish in $\Omega$. Since $g$  is continuous up to the finite boundary of $\Omega$,   $|g|$ is bounded there by $1/(3-e)$ and $\lim_{z \to \infty_{\scriptscriptstyle\C}}g(z)=0$, the maximum principle shows that $g$ is bounded everywhere in $\Omega$ by $1/(3-e)$.
\end{proof}

\begin{corollary}\label{cor:near zeros a few of 1-zN} Let $1<r_1<r_2<\cdots< r_m$, and let $M$ be an integer $M \ge 1$. Then
$$
\Big|\prod_{j=1}^m \Big(1-\Big(\frac{z}{r_j}\Big)^M\Big)\Big|\ge (3-e)^m\, , \quad \mbox{if $\mbox{\rm dist}\big(z, r_j \mathcal{U}_M\big)\ge {r_j}/{M}$ for $1\le j \le m$}\, .
$$
\end{corollary}

We are now ready to describe  the iterative \textbf{construction of $f$}.

\smallskip

Start with $P_0\equiv 1$. Suppose that we have completed $k{-}1$ steps in the
construction of our  power series and that so far we have a section,
denoted $P_{k-1}$, which has degree $d_{k-1}$.

\smallskip

Write the target probability measure $\varphi^{(k)}$ as $\varphi^{(k)}=(1/m) \sum_{j=1}^m  \delta_{r_j}$, where $1<r_1<\cdots< r_m$. Denote
$$\tau=\min\Big\{\frac{r_{j}{-}r_{j{-}1}}{r_j{+}r_{j{-}1}}; 1< j \le m\Big\}$$ and  let $A=\max\big\{|P_{k{-}1}(z)|; \, |z|\le 2 r_m\big\}$.

\smallskip
For integers $N,M\ge 1$, to be determined shortly, we set $$P_{k}(z)=P_{k{-}1}(z)+z^N \prod_{j=1}^m \Big(1-\Big(\frac{z}{r_j}\Big)^M\Big)\,.$$

To start with we require $(\star_1)\, \, N >d_{k{-}1}$. This gives that the coefficients up to index $d_{k{-}1}$ of $P_k$ and of $P_{k{-}1}$ coincide. Observe that the degree $d_k$ of $P_k$ is $d_k=d_{k{-}1}+N+mM$.

Now, we  may choose $N$ large enough, depending only on $m$,  so that if
$$
P_{k}(z)-P_{k{-}1}(z)=\sum_{n=N}^{d_k} b_n z^n\, .
$$
then $\sqrt[n]{b_n}\le 1+\dfrac{1}{k}$, for each $N\le n \le d_k$, no matter what $M\ge 1$ may be.
For that purpose and since $r_j>1$, it is enough to choose $N$ so that
$$
(\star_2)\quad\binom{m}{\lceil m/2\rceil}^{1/N}\le 1+\frac{1}{k}\, .
$$
Observe also that the coefficient of index $N$ of $P_k$ is 1. All this means that the final outcome of this iterative construction will be a power series in $\F$.
We remark that this requirement upon $N$ imposes no restriction on $M \ge 1$.

\smallskip

Next we study the distribution of the zeros of $P_k(z)$ with the aim of showing that the circular projection of the zero counting measure of $P_k$ on the positive real axis is close to the given $\varphi^{(k)}$. We shall compare the location of the zeros of $P_k$ and the location of the zeros  of $z^N \prod_{j=1}^m \big(1-({z}/{r_j})^M\big)$. The zeros of this last polynomial are $\bigcup_{j=1}^m r_j \mathcal{U}_M$ and $z=0$, a total of $N$ times.

We require next that $M$ is so large that $(\flat)\, \,  (1/M)\le \tau$ and also that $(\flat\flat)\,\, r_1(1-1/M)>(1+r_1)/2$. Because of $(\flat)$, the disks $\{z\in \C: |z-r_j \eta|=r_j/M\}$ where $1\le j \le m$ and $\eta \in \mathcal{U}_M$ are pairwise disjoint.

We apply Rouche's theorem in each of these disks. If $z$ is such that $|z-r_j \eta|=r_j/M$ with $1\le j \le m$ and $\eta \in \mathcal{U}_M$, then $|z|\le (1+1/M)r_j\le 2 r_m$ and, therefore,
$$
\Big|P_k(z)-z^N \prod_{j=1}^m \Big(1-\Big(\frac{z}{r_j}\Big)^M\Big)\Big|=|P_{k-1}(z)|\le A\, ,
$$
while, because of Corollary \ref{cor:near zeros a few of 1-zN} and $(\flat\flat)$,
$$
\Big|z^N \prod_{j=1}^m \Big(1-\Big(\frac{z}{r_j}\Big)^M\Big)\Big|\ge \Big(\frac{r_1+1}{2}\Big)^N (3-e)^m\,.
$$

Therefore, if $N$ is such that  $(\star_3)\quad ((r_1+1)/2)^N (3-e)^m >A$, the polynomial  $P_k(z)$ has one zero in each of the disks $\{z \in \C: |z-r_j \eta|=r_j/M\}$. This occurs no matter what the value of $M$ is, as long as $(\flat)$ and $(\flat\flat)$ are satisfied.

Fix $N$ satisfying all the conditions $(\star_i)$ above.

Finally, choose $M$ satisfying, besides $(\flat)$ and $(\flat\flat)$ above, that $
r_m/M\le 1/k$ and that $N+d_{k{-}1}< m M/k$. Lemma \ref{lemma:distance of some probabilities} now gives us that
$$
D(\rho, \varphi^{(k)})\le \frac{1}{k}
$$
where $\rho=\dfrac{1}{d_k}\displaystyle\sum_{w \in \mathcal{Z}_{P_k}} \delta_{|w|}$.

\

We iterate this  construction. The final outcome is a power series $f \in \F$ whose associated probabilities $\rho_n(f)$ satisfy, as desired,
$$D\big(\rho_{d_k}, \varphi^{(k)}\big)\le \frac{1}{k}\, , \quad \mbox{for each $k \ge 1$}\, .$$

\

\

\

\end{document}